\documentclass[10pt,reqno]{amsart}
\usepackage{latexsym,amsmath}
\usepackage{times}
\usepackage{amsfonts}
\usepackage{amssymb}
\usepackage{latexsym}

\usepackage{bbm}

\usepackage{color}
\usepackage{fullpage}

\numberwithin{equation}{section}

\newcommand\vv{\vec v}

\newcommand\nix{\,\cdot\,}

\newcommand\atom{\delta}

\newcommand\thet{\vartheta}

\newcommand\dd{{\mathrm d}}

\newcommand\G{\vec G}
\newcommand\T{\vec T}

\newcommand\br[1]{\left(#1\right)}

\def\vec#1{\mathbf{#1}}

\DeclareMathOperator{\pr}{\mathbb P}

\newtheorem{definition}{Definition}[section]

\newtheorem{theorem}[definition]{Theorem}
\newtheorem{lemma}[definition]{Lemma}
\newtheorem{proposition}[definition]{Proposition}

\newcommand\core{\cC}

\newcommand\cB{\mathcal{B}}
\newcommand\cC{\mathcal{C}}

\newcommand\cF{\mathcal{F}}
\newcommand\cG{\mathcal{G}}
\newcommand\cE{\mathcal{E}}

\newcommand\cT{\mathcal{T}}
\newcommand\cI{\mathcal{I}}

\newcommand\cJ{\mathcal{J}}
\newcommand\cL{\mathcal{L}}

\newcommand\cP{\mathcal{P}}
\newcommand\cX{\mathcal{X}}

\def\cC{{\mathcal C}}
\def\cE{{\mathcal E}}

\newcommand\eps{\varepsilon}

\newcommand\Erw{\mathbb{E}}

\newcommand{\vecone}{\vec{1}}

\newcommand{\Po}{{\rm Po}}
\newcommand{\Bin}{{\rm Bin}}
\newcommand{\Be}{{\rm Be}}

\newcommand\bc[1]{\left({#1}\right)}
\newcommand\cbc[1]{\left\{{#1}\right\}}

\newcommand\brk[1]{\left\lbrack{#1}\right\rbrack}

\newcommand\abs[1]{\left|{#1}\right|}

\newcommand\RR{\mathbb{R}}

\newcommand{\whp}{w.h.p.}

\newcommand{\stacksign}[2]{{\stackrel{\mbox{\scriptsize #1}}{#2}}}

\newcommand{\Erdos}{Erd\H{o}s}
\newcommand{\Renyi}{R\'enyi}

\newcommand{\Bollobas}{Bollob\'as}

\newcommand{\Luczak}{\L uczak}

\newcommand\Lem{Lemma}
\newcommand\Prop{Proposition}
\newcommand\Thm{Theorem}

\newcommand\Sec{Section}
\newcommand\Chap{Chapter}

\newcommand\RSA{Random Structures and Algorithms}
\newcommand\JCTB{Journal of Combinatorial Theory, Series~B}

\newcommand{\toboss}{\uparrow}
\newcommand{\fromboss}{\downarrow}

\newcommand{\labelset}{\{000,001,010,110,111\}}
\newcommand{\Ldn}{\Lambda_{d,n}}

\begin{document}

\title{How does the core sit inside the mantle?}

\author{Amin Coja-Oghlan$^*$, Oliver Cooley$^{**}$, Mihyun Kang$^{**}$ and Kathrin Skubch}
\thanks{$^*$The research leading to these results has received funding from the European Research Council under the European Union's Seventh 
Framework Programme (FP/2007-2013) / ERC Grant Agreement n.\ 278857--PTCC\\
$^{**}$Supported by Austrian Science Fund (FWF): P26826 and W1230, Doctoral Program ``Discrete Mathematics''.\\
An extended abstract for this work has been submitted to EUROCOMB2015.}
\date{\today}

\address{Amin Coja-Oghlan, {\tt acoghlan@math.uni-frankfurt.de}, Goethe University, Mathematics Institute, 10 Robert Mayer St, Frankfurt 60325, Germany.}

\address{Oliver Cooley, {\tt cooley@math.tugraz.at}, Technische Universit\"at Graz, Institute of Optimization and Discrete Mathematics (Math B), Steyrergasse 30, 8010 Graz, Austria}

\address{Mihyun Kang, {\tt kang@math.tugraz.at}, Technische Universit\"at Graz, Institute of Optimization and Discrete Mathematics (Math B), Steyrergasse 30, 8010 Graz, Austria}

\address{Kathrin Skubch, {\tt skubch@math.uni-frankfurt.de}, Goethe University, Mathematics Institute, 10 Robert Mayer St, Frankfurt 60325, Germany.}

\maketitle

\begin{abstract}
\noindent
The $k$-core, defined as the largest subgraph of minimum degree $k$,  of the random graph $\G(n,p)$ has been studied extensively.
In a landmark paper Pittel, Wormald and Spencer [\JCTB\ {\bf 67} (1996) 111--151] determined the threshold $d_k$ for the appearance of an extensive $k$-core.
Here we derive a multi-type branching process that describes precisely how the $k$-core is ``embedded'' into the random graph for any $k\geq3$ and any fixed average degree $d=np>d_k$.
This generalises prior results on, e.g., the internal structure of the $k$-core.

\end{abstract}

\section{Introduction}\label{Sec_intro}

\noindent{\em Let $\G=\G(n,d/n)$ denote the random graph on the vertex set $[n]=\cbc{1,\ldots,n}$
	in which any two vertices are connected with probability $p=\frac dn$ independently.
	Throughout the paper we let $d>0$ be a number that remains fixed as $n\to\infty$.
	The random graph $\G$ enjoys a property with high probability (`\whp') if its probability tends to $1$ as $n\to\infty$.}

\subsection{Background and motivation}\label{Sec_Back}
The``giant component'' has remained a guiding theme in the theory of random graphs
ever since the seminal paper of \Erdos\ and \Renyi~\cite{ER}.
By now, there exists an impressive body of work on its birth, size, avoidance, central and local limits as well as its large deviations  (among other things),
derived via combinatorial, probabilistic and analytic methods~\cite{BB,JLR}.
A key observation in this line of work is
that the emergence of the giant component is analogous to the survival of a Galton-Watson branching process~\cite{Karp}.
This is important not only because this observation leads to a beautiful proof of the original result of \Erdos\ and \Renyi,
but also because the branching process analogy crystallises the interplay of the local and the global structure of the random graph.
Indeed, the notion that the Galton-Watson tree is the limiting object of the ``local structure'' of the random graph can be formalised
neatly in the language of ``local weak convergence''~\cite{Aldous,BenjaminiSchramm,BordenaveCaputo}.

Because for any $k\geq3$ the $k$-core, defined as the (unique) maximal subgraph of minimum degree $k$,
is identical to the largest $k$-connected subgraph of the random graph \whp~\cite{Tomasz1,Tomasz2},
the $k$-core is perhaps the most natural generalisation of the ``giant component''.
As a consequence, the $k$-core problem has attracted a great deal of attention.
Pittel, Wormald and Spencer~\cite{Pittel} were the first to determine the precise threshold $d_k$ beyond which the $k$-core is non-empty \whp\ 
Indeed, they obtained a formula for its asymptotic size.
Specifically, denote by $\cC_k(G)$ the $k$-core of a graph $G$.
Then for any $k\geq3$ there is a function $\psi_k:(0,\infty)\to[0,1]$ such that for any $d\in(0,\infty)\setminus\cbc{d_k}$ the sequence
$(n^{-1}|\cC_k(\G)|)_n$ converges to $\psi_k(d)$ in probability.
The function $\psi_k$ is identical to $0$ for $d<d_k$, continuous, strictly increasing and strictly positive for $d>d_k$ but, remarkably, discontinuous at the point $d_k$:
	 the moment the $k$-core emerges, it is of linear size~\cite{Tomasz1,Tomasz2}.
The proof in~\cite{Pittel} is based on a careful study of a ``peeling process'' that repeatedly removes vertices of degree less than $k$ from the random graph.
However, Pittel, Wormald and Spencer pointed out that a simple ``branching process'' heuristic predicts the correct threshold
and the correct size of the $k$-core, and this argument has subsequently been turned into an alternative proof of their result~\cite{MolloyCores,Riordan}.

The aim of the present paper is to enhance this branching process perspective of the $k$-core problem
to characterise how the $k$-core ``embeds'' into the random graph.
More specifically, we are concerned with the following question.
Fix $k\geq3$, $d>d_k$ and let $s>0$ be an integer.
Generate a random graph $\G$ and colour each vertex that belongs to the $k$-core black and all other vertices white.
Now, pick a vertex $\vec v$ uniformly at random.
What is the distribution of the {\em coloured} subgraph
induced on the set of all vertices at distance at most $s$ from $\vec v$?
Of course, {\em without} the colours the standard branching process analogy yields convergence to the
``usual'' Galton-Watson tree with $\Po(d)$ offspring.
The point of the present paper is to exhibit a multi-type branching process that yields the limiting distribution of the {\em coloured} subgraph.
In particular, this process describes exactly how we walk into and out of the $k$-core while exploring the random graph from $\vec v$.

This is challenging because
the distribution of the interconnections between the $k$-core and the ``mantle'' (i.e., the vertices outside the core) is intricate.
For instance, suppose that for each vertex $v$ we are given the number  $d^*(v)$  of neighbours that $v$ has inside the core of $\G$
and the number $d_*(v)$  of neighbours that $v$ has in the mantle.
Then the core is simply equal to the set $S$ of all vertices $v$ such that $d^*(v)\geq k$.
But if, conversely,  we sample a graph $\G'$ randomly subject to the condition that every vertex $v$ has $d^*(v)$ neighbours in $S$
and $d_*(v)$ neighbours outside of $S$, then \whp\ the core of $\G'$ will {\em not} be identical to $S$.
In fact, the $k$-core of $\G'$ is a proper superset of $S$ \whp\
One reason for this is that \whp\ there will be $\Omega(n)$ vertices $v\not\in S$ such that $d^*(v)=k-1$ and $d_*(v) \ge 1$.
Consequently, \whp\ there will be two such vertices $v,v'$ that are adjacent in $\G'$ and that therefore belong to its $k$-core.

\subsection{Results}\label{Sec_results}
Recall that a (possibly infinite) graph is {\em locally finite} if all vertices have finite degree.
By a \emph{rooted graph} we mean a connected locally finite graph $G$ on a countable vertex set together with a distinguished vertex $v_0\in V(G)$, the {\em root}.
If $\cX$ is a finite set, then 
a \emph{rooted $\cX$-marked graph} is a rooted graph $G$ together with a map $\sigma:V(G)\to\cX$.
Two rooted $\cX$-marked graphs $(G,v_0,\sigma)$, $(G',v_0',\sigma')$ are {\em isomorphic}
if there is an isomorphism $\pi:G\to G'$ of the graphs $G,G'$ such that $\pi(v_0)=v_0'$ and $\sigma=\sigma'\circ\pi$.
Let $[G,v_0,\sigma]$ denote the isomorphism class of $(G,v_0,\sigma)$ and let
$\cG_\cX$ be the set of all isomorphism classes of rooted $\cX$-marked graphs.
Further, for $s\geq0$ let $\partial^s[G,v_0,\sigma]$ signify the isomorphism class
of the (finite) rooted $\cX$-marked graph obtained 
by deleting all vertices at a distance greater than $s$ from $v$. We sometimes omit the arguments $v_0$ and $\sigma$ when they are clear from the context.

If $G$ is a graph and $k\geq3$ is an integer, then
	$\sigma_{k,G}:V(G)\to\cbc{0,1}$, $v\mapsto\vecone\cbc{v\in\cC_k(G)}$
indicates membership of the $k$-core.
Further, for a vertex $v$ of $G$ we let
$G_v$ denote the component of $v$. 
Then $(G_v,v,\sigma_{k,G_v})$ is a rooted $\cbc{0,1}$-marked graph whose marks indicate membership of the $k$-core of the component $G_v$.
Our aim is to determine the distribution of $\{\partial^s[\G,v,\sigma_{k,\G_{v}}] \, : \, v\in V(G) \}$.

To this end, we construct a multi-type branching process 
that generates a (possibly infinite) rooted $\cbc{0,1}$-marked tree.
As we saw in \Sec~\ref{Sec_intro}, the connections between the mantle and the $k$-core are subject to non-trivial correlations.
Therefore, it might seem remarkable that there even exists a branching process that captures the local structure of $\G$ marked according to $\sigma_{k,\G}$.
The solution is that the branching process actually possesses more than two types.
Indeed, there are five different vertex types, denoted by $(0,0,0)$, $(0,0,1)$, $(0,1,0)$, $(1,1,0)$, $(1,1,1)$. To simplify the notation, we will often write $000$ instead of $(0,0,0)$, and similarly for all other types.
The mark of every vertex will simply be the first ``bit'' of its type.
In other words, the $\cbc{0,1}$-marked random tree that we create is actually the projection of an enhanced tree that
contains the necessary information to accommodate the relevant correlations.

Apart from $d,k$, the $5$-type branching process $\hat\T(d,k,p)$ has a further parameter $p\in[0,1]$. 
Setting
	\begin{align}\label{type_prob}
	q=q(d,k,p)&:=\pr\brk{\Po(dp)=k-1|\Po(dp)\geq k-1},
	\end{align}
we define
	\begin{align*}
	p_{000}&:=1-p,&
	p_{010}&:=pq,&
	p_{110}&:=p(1-q).
	\end{align*}
The process starts with a single vertex $v_0$, whose type is chosen from $\{000,010,110\}$ according to the distribution $(p_{000},p_{010},p_{110})$.
Subsequently, each vertex of type $z_1z_2z_3 \in \labelset$ spawns a random number of vertices of each type.
The offspring distributions are defined by the generating functions $g_{z_1z_2z_2}(\vec x)$ detailed in Figure~\ref{Fig_g}, 
 where
	$\vec x=(x_{000},x_{001},x_{010},x_{110},x_{111})$ and
\[
\bar{q}=\bar{q}(d,k,p) :=\pr\brk{\Po(dp)=k-2|\Po(dp)\leq k-2}.
\]
Thus, for an integer vector $\vec y=(y_{000},y_{001},y_{010},y_{110},y_{111})$ the probability that a vertex of type
$z_1z_2z_3$ generates offspring $\vec y$ equals the coefficient of
	the monomial $x_{000}^{y_{000}}\cdots x_{111}^{y_{111}}$ in $g_{z_1z_2z_3}(\vec x)$.
\begin{figure}\small
	\begin{align*}
	g_{000}(\vec x)	
		&=\exp(d(1-p)x_{000})\frac{\sum_{h=0}^{k-2}(dp)^{h}(qx_{010}+(1-q)x_{110})^{h}/h!}{\sum_{h=0}^{k-2}(dp)^h/h!},\\
	g_{001}(\vec x)
		&=\bar{q}\br{\exp(d(1-p)x_{001})\br{qx_{010}+(1-q)x_{110}}^{k-2}}\\
		&\qquad+(1-\bar{q})\br{\exp(d(1-p)x_{000})\frac{\sum_{h=0}^{k-3}(dp)^{h}(qx_{010}+(1-q)x_{110})^{h}/h!}{\sum_{h=0}^{k-3}(dp)^h/h!}},\\
	g_{010}(\vec x)
		&=\exp(d(1-p)x_{001})\br{qx_{010}+(1-q)x_{110}}^{k-1},\\
	g_{110}(\vec x)
		&=\exp(d(1-p)x_{001})	\frac{\sum_{h\geq k}(dpx_{111})^{h}/h!}{\sum_{h\geq k}(dp)^h/h!},\\
	g_{111}(\vec x)
		&=\exp(d(1-p)x_{001})\frac{\sum_{h\geq k-1}(dpx_{111})^{h}/h!}{\sum_{h\geq k-1}(dp)^h/h!}.
	\end{align*}
\caption{The generating functions $g_{z_1z_2z_3}(\vec x)$.}\label{Fig_g}
\end{figure}

Finally, we turn the resulting $5$-type random tree into a $\cbc{0,1}$-marked tree rooted at $v_0$ by giving mark $0$ to 
all vertices of type $000$, $001$ or $010$, and mark $1$ to vertices of type $110$ or $111$.
Let $\T(d,k,p)$ 
signify  the resulting (possibly infinite) random rooted $\{0,1\}$-marked tree, i.e.\ $\T(d,k,p)$ is a $2$-type projection of the $5$-type process $\hat\T(d,k,p)$.

\begin{theorem}\label{Thm_main}
Assume that $k\geq3$ and $d>d_k$.
Let $s\geq0$ be an integer and let $\tau$ be a rooted $\cbc{0,1}$-marked tree.
Moreover, let $p^*$ be the largest fixed point of
	\begin{equation}\label{eqmain}
	\phi_{d,k}:[0,1]\to[0,1],\qquad p\mapsto\pr\brk{\Po(dp)\geq k-1}.
	\end{equation}
Then 
	$$\frac1n\sum_{v\in V(\G)}
		\vecone\cbc{\partial^s[\G,v,\sigma_{k,\G_{v}}]=\partial^s[\tau]}$$
converges to $\pr\brk{\partial^s[\T(d,k,p^*)]=\partial^s[\tau]}$ in probability.
\end{theorem}

\noindent
In words, \Thm~\ref{Thm_main} states that  \whp\ 
the fraction
of vertices $v$ whose depth-$s$ neighbourhood in $\G$ marked according to $\cC_k(\G)$ is isomorphic to $\tau$ 
is asymptotically equal to 
 the probability that the random marked tree $\T(d,k,p^*)$ truncated after $s$ generations is isomorphic to $\tau$.
The proof of \Thm~\ref{Thm_main} will reveal the origin of the generating functions from Figure~\ref{Fig_g}.
They derive from a systematic understanding of the correlations that determine the connections between the mantle and the core.

\Thm~\ref{Thm_main} 
 can be cast elegantly in the framework of local weak convergence;
	the concrete formulation that we use resembles that employed in~\cite{BordenaveCaputo}.
More specifically, we endow the space $\cG_{\cX}$ of isomorphism classes of rooted $\cX$-marked graphs with the coarsest topology that makes all the maps
	\begin{align}\label{eqtop}
	\chi_{\Gamma,s}:\cG_\cX\to\cbc{0,1},&\qquad \Gamma'\mapsto\vecone\cbc{\partial^s\Gamma=\partial^s\Gamma'}&
		\hspace{-1cm}(\Gamma\in\cG_{\cX},s\geq0)
	\end{align}
continuous.	
Let $\cP(\cG_{\cX})$ denote the space of probability measures on $\cG_{\cX}$ equipped with the weak topology.
For $\Gamma\in\cG_{\cX}$ let $\atom_{\Gamma}\in\cP(\cG_{\cX})$ signify the Dirac measure on $\Gamma$.
Further, let
$\cP^2(\cG_{\cX})
	$ be the space of probability measures on $\cP(\cG_{\cX})$, also with the weak topology.
For $\lambda\in\cP(\cG_{\cX})$ let $\atom_\lambda\in\cP^2(\cG_{\cX})$ be
	the Dirac measure on $\lambda$.
Then any graph $G$ gives rise to a measure
	$$\lambda_{k,G}:=\frac1{|V(G)|}\sum_{v\in V(G)}\delta_{[G_v,v,\sigma_{k,G_v}]}\in\cP(\cG_{\cbc{0,1}}),$$
which is simply the empirical distribution of ``marked neighbourhoods'' of the vertices $v\in V(G)$. 
Hence,
	\begin{equation}\label{eqLambdadkn}
	\Lambda_{d,k,n}:=\Erw_{\G}[\delta_{\lambda_{k,\G}}]\in\cP^2(\cG_{\cbc{0,1}})
	\end{equation}
captures the distribution of the ``marked neighbourhoods'' of the random graph $\G$.
Additionally, let $\cL[\T(d,k,p)]\in\cP(\cG_{\{0,1\}})$ denote the distribution of the isomorphism class of the random tree $\T(d,k,p)$.
Finally, let
	$$\thet_{d,k,p}=\cL[\T(d,k,p)]\in\cP(\cG_{\{0,1\}}).$$

\begin{theorem}\label{Thm_lwc}
Assume that $k\geq3$ and $d>d_k$ and let $p^*$ be the largest fixed point of the function  $\phi_{d,k}$ from (\ref{eqmain}).
Then $\lim_{n\to\infty}\Lambda_{d,k,n}=\atom_{\thet_{d,k,p^*}}$.
\end{theorem}

\noindent
We shall derive  \Thm~\ref{Thm_main} from \Thm~\ref{Thm_lwc}.
Conversely, it is not difficult to see that \Thm~\ref{Thm_main} implies \Thm~\ref{Thm_lwc}.

\subsection{Related work}
The $2$-core exhibits qualitatively different behaviour to the $k$-core for $k\geq3$.
For instance, the $2$-core is non-empty with a non-vanishing probability for any $d>0$.
Moreover, it is of size $\Omega(n)$ for any $d>1$  \whp\ 
Thus, the ``giant $2$-core'' emerges alongside (in fact mostly inside) the giant component.
While we omit a detailed discussion of the literature on the $2$-core (some of the references can be found in~\cite{BB,JLR}),
we remark that several of the arguments developed for the study of the $k$-core for $k\geq3$ encompass the case $k=2$ as well.
>From now on, we will always assume that $k\geq3$.

Since the work of Pittel, Wormald and Spencer~\cite{Pittel}
several different arguments for determining the location of the $k$-core for $k\geq3$ have been suggested.
Some of these approaches have advantages over~\cite{Pittel}, such as being simpler, giving additional information, or applying to a broader class of models
	or combinatorial structures (e.g., hypergraphs rather than just graphs, or random graphs/hypergraphs with given degree sequences).
Roughly, there are two types of proofs.
First, approaches that rely on the analysis of a peeling process akin to the one of Pittel, Wormald and
	Spencer~\cite{Encores,Cooper,Fernholz1,Fernholz2,JansonLuczak,JansonLuczak2,Kim}.
Among these~\cite{Encores,Kim} stand out as they characterise the distribution of the $k$-core and thus
make it amenable to an analysis by standard random graph techniques.
Second, arguments that formalise the ``branching process'' intuition~\cite{MolloyCores,Riordan} put forward in~\cite{Pittel}. In~\cite{Darling} this was achieved in a general non-uniform hypergraph setting via a differential equations method, but conditioned on the degree sequence.

Some results on the local structure of the core and the mantle follow directly from the aforementioned analyses.
For instance, the Poisson cloning model~\cite{Kim} immediately implies that the {\em internal} \ local structure of the $k$-core
can be described by a simple (single-type) Galton-Watson process.
Riordan also pointed out that this local description follows from his analysis~\cite{Riordan}.
Furthermore, Cooper~\cite{Cooper} derived the asymptotic distribution of the internal and the external degree sequences of the vertices in the mantle,
i.e., of the number of vertices with a given number of neighbours in the core and a given number of neighbours outside.
In addition, Sato~\cite{Sato} studied the robustness of the $k$-core (i.e., the impact of deleting random edges).

The contribution of the present work is that we exhibit a branching process that describes
the structure of the core together with the mantle and hence, crucially, the connections between the two.
This is reflected in the fact that \Thm s~\ref{Thm_main} and~\ref{Thm_lwc} deal with $\cbc{0,1}$-marked graphs and trees,
with marks indicating membership of the $k$-core.
Neither the construction of the core via the ``peeling process'' nor the branching process analogy from~\cite{MolloyCores,Pittel,Riordan}
reveal how the core interconnects with the mantle.
In fact, even though~\cite{Cooper} asymptotically determines the degree distribution of the core along with the combined degrees of the vertices in the mantle,
we saw in \Sec~\ref{Sec_Back} that the conditional random graph is {\em not} uniformly random subject to these.

Structures that resemble cores of random (hyper)graphs have come to play an important role in the study of random constraint satisfaction problems,
	particularly in the context of the study of the ``solution space''~\cite{Barriers,Molloy}.
This was first noticed in non-rigorous but analytic work based on ideas from statistical physics (see~\cite{MM} and the references therein).
Indeed, in the physics literature it was suggested to characterise the core by means of a ``message passing'' algorithm called
{\em Warning Propagation}~\cite[\Chap~18]{MM}.
A similar idea is actually implicit in Molloy's paper~\cite[proof of \Lem~6]{Molloy}.
The Warning Propagation description of the core will play a key role in the present paper,
as we shall explain in the next subsection.

\subsection{Techniques and outline}
There is a very natural formulation of Warning Propagation to identify the $k$-core of a given graph $G$.
It is based on introducing ``messages'' on the edges of $G$ and marks on the vertices of $G$, both with values in $\{0,1\}$.
These will be updated iteratively in terms of  a ``time'' parameter $t\geq 0$.
At time $t=0$ we start with the configuration in which all messages are equal to $1$, i.e.,
	 \begin{equation}\label{initialisation}
	    \mu_{v\to w}(0|G)=1\qquad\mbox{for all }\cbc{v,w}\in E(G).
	 \end{equation}
Inductively, writing $\partial v$ for the neighbourhood of vertex $v$ and abbreviating $\partial v\setminus w=\partial v\setminus\cbc w$, we let
	\begin{equation}\label{messages}
	  \mu_{v\to w}(t+1|G)=\vecone\cbc{\sum_{u\in\partial v\setminus w}\mu_{u\to v}(t|G)\geq k-1}.
	\end{equation}
The messages are directed, i.e.\ at each time there are {\em two} messages $\mu_{v\to w}(t|G)$, $\mu_{w\to v}(t|G)$ travelling along the edge $\cbc{v,w}$.
Additionally, the mark of $v\in [n]$ at time $t$ is
	\begin{equation}\label{marks}
	  \mu_v(t|G)=\vecone\cbc{\sum_{u\in\partial v}\mu_{u\to v}(t|G)\geq k}.
	\end{equation}
The intuition is that if $\mu_v(t|G)=0$, then by time $t$ Warning Propagation has ruled out that $v$ belongs to the $k$-core of $G$. 
Conversely, we shall see that the messages converge to a fixed point for any $G$, and that the set of vertices marked $1$ in the fixed point coincides with the $k$-core
	(see Lemma~\ref{Lemma_WP} below).

Indeed,  in \Sec~\ref{Sec_WP} we are going to see that in the case of the random graph $\G$,
a {\em bounded} number of iterations 
suffice to obtain an accurate approximation of the $k$-core \whp\
More precisely, for any fixed $\eps>0$ there exists an integer $t>0$ such that $S=\cbc{v \in [n]:\mu_v(t|\G)=1}$ is
a superset of the $k$-core $\core_k(\G)$ such that $|S\setminus \cC_k(\G)|\le \eps n$ \whp\ (see Lemma~\ref{Prop_bottomUp}).
While this is already implicit in Molloy's proof~\cite{MolloyCores}, he does not phrase it in the language of Warning Propagation
and therefore we provide a (simple) self-contained derivation. 
Based on this observation we will argue that the analysis of the Warning Propagation fixed point on $\G$
reduces to the study of Warning Propagation on the (infinite) Galton-Watson tree with $\Po(d)$ offspring, which is the main result of \Sec~\ref{Sec_WP}.

In \Sec~\ref{Sec_Kathrin}
we show that the multi-type branching process from \Sec~\ref{Sec_results} describes the
distribution of the Warning Propagation fixed point on the infinite $\Po(d)$ Galton-Watson tree.
The somewhat delicate proof of this fact requires several steps.
The key one is to turn the problem of tracing how Warning Propagation passes
messages from the ``bottom'' of the Galton-Watson tree up toward the root into a process where messages are passed ``top-down'', i.e.,
in the fashion of a branching process.
But before we come to this, we need to introduce some background and notation.

\section{Preliminaries}

\subsection{Notation}
Throughout the paper all graphs are assumed to be locally finite with a countable vertex set.
For a graph $G$ and a vertex $v\in V(G)$ we denote 
by $\partial^s(G,v)$ the subgraph of $G$ induced on the set of vertices at distance at most $s$ from $v$. 
We abbreviate $\partial (G,v) = \partial^1(G,v)$ and $\partial v = V(\partial (G,v) - v)$ whenever $G$ is clear from the context.

By a {\em rooted graph} we mean a connected locally finite graph $G$ together with a root $v_0$.
A {\em child} of $v\in V(G)$ is a vertex $w\in\partial v$ whose distance from $v_0$ is greater than that of $v$.
For a vertex $v$ of a rooted graph $G$ we denote by $\partial_+v$ the set of all children of $v$.

Two rooted graphs $(G,v_0)$, $(G',v_0')$ are {\em isomorphic} if there is an isomorphism $\pi:G\to G'$ such that $\pi(v_0)=\pi(v_0')$.
Write $[G,v_0]$ for the isomorphism class of $(G,v_0)$ and let $\cG$ be the set of all isomorphism classes of rooted graphs.
Further, let $\partial^s[G,v_0]$ be the isomorphism class of the rooted graph obtained
 from $(G,v_0)$  by removing all vertices at distance greater than $s$ from $v_0$. We sometimes omit the argument $v_0$ when the root is clear from the context.

For a random variable $X:\Omega\to\cE$ with values in a space $\cE$ we let $\cL(X)$ denote the distribution of $X$.
Thus, $\cL(X)$ is a probability measure on the space $\cE$.

For real numbers $y,z>0$ we let $\Po_{\geq z}(y)$ denote the Poisson distribution $\Po(y)$ conditioned on the event that the outcome is at least $z$.
Thus, 
	$$\pr\brk{\Po_{\geq z}(y)=x}=\frac{\vecone\cbc{x\geq z}y^x\exp(-y)}{x!\pr\brk{\Po(y)\geq z}}\qquad\mbox{for any integer $x\geq0$}.$$
The distributions $\Po_{> z}(y)$, $\Po_{\leq z}(y)$, $\Po_{< z}(y)$ are defined analogously.

We will need to define various random objects and distributions in the paper, with subtle differences between them. \textbf{Reference tables are provided in an appendix} to help the reader maintain an overview of all the definitions.

\subsection{Local weak convergence}
In what follows it will be more convenient
to work with random trees instead of isomorphism classes. Therefore from now on we assume that all random isomorphism classes of graphs, such as branching processes, are additionally equipped with almost surely distinct random labels on every vertex 
to obtain a representative of the particular class. This distinction will
be technically necessary in upcoming proofs. However the actual label of a vertex in this sense will not be taken into account in any of them. 
Thus, let $\T(d)$ denote the random tree which is obtained by labelling each vertex in the standard (single-type) Galton-Watson tree 
with offspring distribution $\Po(d)$
with a number in $[0,1]$ independently and uniformly at random. 

It is well-known that the ``local structure'' of the random graph $\G$ converges to $[\T(d)]$.
We formalise this statement via local weak convergence, closely following~\cite{BordenaveCaputo}.
Thus, we endow $\cG$ with the coarsest topology such that for any $\Gamma\in\cG$, $s\geq0$ the map
	$\chi_{\Gamma,s}:\Gamma'\in\cG\mapsto\vecone\cbc{\partial^s\Gamma'=\partial^s\Gamma}$ is continuous.
Let $\cP(\cG)$ be the set of all probability measures on $\cG$ and let $\cP^2(\cG)$ be the set of all probability measures on $\cP(\cG)$.
Both of these spaces carry the weak topology.

Any graph $G$ gives rise to the empirical distribution
$$\lambda_G=\frac1{|V(G)|}\sum_{v\in V(G)}\atom_{[G_v,v]}\in\cP(\cG).$$
Further, let 
	\begin{equation}\label{eqLambdadn}
	\Ldn=\Erw_{\G}[\delta_{\lambda_{\G}}]\in\cP^2(\cG)
	.\end{equation}
The following theorem expresses the well-known fact that $[\T(d)]$ mirrors the local structure of $\G$
in this notation.

\begin{theorem}\label{Thm_GW}
For any $d>0$ we have 
$\lim_{n\to\infty}\Ldn=\atom_{\cL([\T(d)])}$.
\end{theorem} 
The spaces $\cG,\cP(\cG),\cP^2(\cG)$ are well-known to be Polish, i.e., complete, separable and metrizable.
Analogously, for any finite set $\cX$ the spaces $\cG_\cX,\cP(\cG_\cX),\cP^2(\cG_\cX)$ are Polish.

\subsection{The $k$-core threshold}
We build upon the following result which determines the $k$-core threshold and the asymptotic number of vertices in the $k$-core. Recall that $\G=\G(n,d/n)$.

\begin{theorem}[{\cite{Pittel}}]\label{Thm_PWS}
Let $k\geq3$.
The function $\lambda\in(0,\infty)\mapsto\lambda/\pr\brk{\Po(\lambda)\geq k-1}$ is continuous, tends to infinity as $\lambda\to0$ or $\lambda\to\infty$,
and has a unique local minimum, where it attains the value
	$$d_k=\inf\cbc{\lambda/\pr\brk{\Po(\lambda)\geq k-1}:\lambda>0}.$$
Furthermore, for any $d>d_k$ the equation $d=\lambda/\pr\brk{\Po(\lambda)\geq k-1}$ has precisely two solutions.
Let $\lambda_k(d)$ denote the larger one and define
	\begin{align*}
	\psi_k&:(d_k,\infty)\to(0,\infty),\qquad d\mapsto\pr\brk{\Po(\lambda_k(d))\geq k}.
	\end{align*}
Then $\psi_k$ is a strictly increasing continuous function.
Moreover, if $d<d_k$, then $\cC_k(\G)=\emptyset$ \whp,
while $\frac1n|\cC_k(\G)|$ converges to $\psi_k(d)$ in probability for $d>d_k$.
\end{theorem}

We will use the following lemma to establish a connection between Warning Propagation and the formula for the size of the $k$-core from \Thm~\ref{Thm_PWS}.
Although similar statements are implicit in~\cite{MolloyCores,Riordan}, we include the simple proof for the sake of completeness.

\begin{lemma}\label{Prop_fix}
Suppose $d>d_k$
and let $p^*$ be the largest fixed point of 
the function $\phi_{d,k}$ from (\ref{eqmain}).
Then $\phi_{d,k}$ is contracting on $[p^*,1]$.
Moreover,
	$\psi_k(d)=\pr\brk{\Po(dp^*)\geq k} = \phi_{d,k+1}(p^*).$
\end{lemma}
\begin{proof}
Let 
$\varphi_k(x)=
\pr \brk{\Po(x)\ge k-1}$
for $x\geq 0$. Then 
$\phi_{d,k}(p)=\varphi_k(dp).$
Moreover, $d=\lambda/\varphi_k(\lambda)$ iff $\phi_{d,k} (\lambda/d)=\lambda/d$,
i.e.\ $\lambda$ is a solution to the equation $d=\lambda/\varphi_k(\lambda)$ iff $\lambda/d$ is a fixed point of $\phi_{d,k}$.
Since $p^*=p^*(d,k)$ is the largest fixed point of $\phi_{d,k}$, it holds that
$p^*=\lambda_k(d)/d,$
whence \Thm~\ref{Thm_PWS} entails
$$\psi_k(d)=\pr\br{\Po(\lambda_k(d))\geq k}=\pr\br{\Po(dp^*)\geq k} = \phi_{d,k+1}(p^*).$$

We next show that $p^*>0$ for $k\geq 3$ and $d>d_k$, which is equivalent to $\lambda_k(d) > 0$.
By Theorem~\ref{Thm_PWS}, 
$ d_k = \inf \left\{\left.\lambda/\varphi_k(\lambda)\right|\lambda>0\right\}.$
Note that $\lambda_k(d)$ is a strictly increasing function in $d$
and so for $d > d_k$ we find
$\lambda_k(d)>\lambda_k(d_k) \ge 0 $
as required.

We next aim to prove that $\phi_{d,k}$ is a contraction on $[p^*,1]$. We consider the derivatives of $\varphi_k$:
 	\begin{align*}
 	\frac{\partial}{\partial x}\varphi_k(x)&=\frac{1}{(k-2)!\exp(x)}x^{k-2}\geq0,&
	\frac{\partial^2}{\partial x^2}\varphi_k(x)&=\frac{1}{(k-2)!\exp(x)}(-x+k-2)x^{k-3}.
	\end{align*}
Using
$\frac{\partial^i}{\partial p^i}\phi_{d,k}(p)=d^i\frac{\partial^i}{\partial x^i}\left.\varphi_k(x)\right|_{x=dp}$
we obtain
	\begin{align*}
	\frac{\partial}{\partial p}\phi_{d,k}(p)&\geq 0\qquad\mbox{for all $p\in[0,1]$,}&
	{\rm sgn}\left(\frac{\partial^2}{\partial p^2}\phi_{d,k}(p)\right)&= {\rm sgn}\left( \frac{k-2}{d}-p \right).
	\end{align*}
Hence, $\partial^2\phi_{d,k}/\partial p^2$ has only one root in $[0,1]$ and $\phi_{d,k}$ is convex on $[0,(k-2)/d]$ 
and concave on $((k-2)/d,1]$. Together with  $\phi_{d,k}(1)<1$ and the fact that $p^* >0$, this implies that one of two possible cases can occur:
$\phi_{d,k}$ either has (apart from the trivial fixed point at $0$) one additional fixed point $p^*>0$, where $\phi_{d,k}$ is tangent to the identity, or two $p^*>p_1>0$, where
$\phi_{d,k}$ crosses the identity. For the purposes of this proof, the latter case is essentially equivalent to a third possible case, namely that $\phi_{d,k}$ has derivative greater than (or equal to) one at the point $p=0$, and so is initially greater than the identity function, but crosses it at point $p^*(d,k)$.

We claim that the first case does not occur for $d>d_k$. Suppose it does and let $d'$ be such that $d_k < d' < d$. Then 
$\phi_{d',k}(p)< \phi_{d,k}(p) \le p$
for all $p>0$. However, this means that $p^*(d',k)=0$, contradicting what we have already proved.
Thus we may assume that the second case holds. Then it follows from $\phi_{d,k}(1)<1$ and the fact that 
$\partial^2\phi_{d,k}/\partial p^2$ changes its sign only once in $[0,1]$, that $\phi_{d,k}$ is concave on $[p^*,1].$ Since $\phi_{d,k}$ is concave 
and monotonically increasing on $[p^*,1]$ its derivative is less than one and $\phi_{d,k}$ is contracting on this interval. 
\end{proof}

\noindent
{\bf\em From now on we assume that $k\geq3$ and that $d>d_k$.
	Further, $p^*$ signifies the largest fixed point of (\ref{eqmain}).}

\section{Warning Propagation and the $k$-core}\label{Sec_WP}

\subsection{Convergence to the $k$-core}
The aim in this section is to reduce the study of the $k$-core on the random graph $\G=\G(n,d/n)$
to the investigation of Warning Propagation on the Galton Watson tree $\T(d)$. 
We start with the following simple observation that strongly resembles the application of Warning Propagation to the $k$-XORSAT problem
	(cf.~\cite[Chapter~19]{MM}).

\begin{lemma}\label{Lemma_WP}
Let $G$ be a locally finite graph.
\begin{enumerate}
\item If $v,w$ are adjacent and $t\geq0$, then $\mu_{v\to w}(t+1|G)\leq\mu_{v\to w}(t|G)$.
	Moreover, $\mu_v(t+1|G)\leq\mu_v(t|G)$ for all vertices $v$.
\item For any $t\geq0$ we have that the $k$-core $\core_k(G)\subset\cbc{v\in V(G):\mu_v(t|G)=1}$.
\item For any vertex $v$ the limit \ $\lim_{t\to\infty}\mu_v(t|G)$ exists
	and $v\in \core_k(G)$ iff \ $\lim_{t\to\infty}\mu_v(t|G)=1$.
\end{enumerate}
\end{lemma}
\begin{proof}
The first claim follows from a simple induction on $t$: If $\mu_{v\to w}(t|G)\leq\mu_{v\to w}(t-1|G)$ for all $\{v,w\} \in E(G)$ and for a fixed $t\geq 0$, we obtain
	$\mu_{v\to w}(t+1|G)
	  \leq\mu_{v\to w}(t|G)$
by (\ref{messages}). In particular this implies that $\mu_v(t+1|G)\leq\mu_v(t|G)$ for all $v\in [n]$ by~\eqref{marks}.

To obtain (2), we  show by induction that, in fact, any vertex $v$ in the core satisfies
\begin{align*}
A(t): & \hspace{0.2cm} \mu_{v\to w}(t|G)=1 \hspace{0.2cm} \mbox{for all } w\in \partial(G,v)\\
B(t): & \hspace{0.2cm} \mu_v(t|G) =1
\end{align*}
for all $t\ge 0$. Property $A(0)$ is certainly true because of our starting conditions. We will show that for all $t\ge 0$, $A(t)$ implies $B(t)\wedge A(t+1)$. This follows because $v$ has at least $k$ neighbours $u$ in the core, all with $\mu_{u\to v}(t|G)=1$ if $A(t)$ holds,
and so~\eqref{marks} implies that $B(t)$ holds while~\eqref{messages} implies that $A(t+1)$ holds.

The first assertion of~(3) holds because $(\mu_v(t))_t$ is decreasing by (1) and bounded by definition. For the second assertion, let $G'$ be the subgraph induced by the vertices with $\lim_{t\to \infty}\mu_v(t|G)=1$. We need to show that $G' \subset \core_k(G)$. Since $G$ is locally finite, a vertex $v$ has a finite number of neighbours
and so there is a time $t_0$ such that for all neighbours $u$ of $v$ and for all times $t\ge t_0$,
the messages from $u$ to $v$ and from $v$ to $u$ and the marks on $u$ and $v$ remain constant.
Since $\mu_v(t_0|G)=1$, by~\eqref{marks} we have $\sum_{u \in \partial v}\mu_{u\to v}(t_0|G) \ge k$
and therefore by~\eqref{messages}, $\mu_{v\to u}(t_0+1|G)=1$ for all $u \in \partial v$.
This means that for all neighbours $u$ with $\mu_{u\to v}(t_0+1|G)=1$, by~\eqref{marks} we have $\mu_u(t_0+1|G)=1$.
Therefore such vertices are in $G'$ and there are at least $k$ of them. This shows that $v$ has degree at least $k$ in $G'$.
Since $v$ was arbitrary, this shows that $G'$ has minimum degree at least $k$ and therefore $G' \subset \core_k(G)$ as required.
\end{proof}

\subsection{Warning Propagation on trees}\label{Sec_fix}
We proceed by relating Warning Propagation on the random tree $\T(d)$ to the fixed point problem from \Lem~\ref{Prop_fix} and thus to \Thm~\ref{Thm_PWS}.
Assume that $(T,v_0)$ is a rooted locally finite tree.
Then a vertex $v\neq v_0$ has a parent $u$ (namely, the neighbour of $v$ on the path to $v_0$).
We use the shorthand
	$$\mu_{v\toboss}(t|T)=\mu_{v\to u}(t|T).$$
Furthermore, for $t\ge 0$ we set
	$$\mu_{v_0\toboss}(0|T)=1\quad\mbox{and}\quad
		\mu_{v_0\toboss}(t+1|T)=\vecone\cbc{\sum_{w\in\partial v_0}\mu_{w\toboss}(t|T)\geq k-1},$$
i.e.\ $\mu_{v_0\toboss}(t|T)$ is the message that $v_0$ would send to its parent if it had one.

\noindent
We observe that on the random tree $\T(d)$ the limit $\mu_{v_0\toboss}(t|\T(d))$ as $t\to\infty$ exists almost surely.

\begin{lemma}\label{Lemma_as}
The sequence $(\mu_{v_0\toboss}(t|\T(d)))_{t\geq1}$ converges almost surely and thus in $L^1$ to a random variable $\mu^*(\T(d))\in\cbc{0,1}$
whose expectation is equal to $p^*$.
\end{lemma}
\begin{proof}
Let $p^{(t)}=\Erw_{\T(d)}[\mu_{v_0\toboss}(t|\T(d))]$ for $t\geq0$ (so in particular $p^{\bc 0}=1$).
The random tree $\T(d)$ is recursive in the following sense: the trees pending on the children of the root
are distributed as independent copies of the random tree $\T(d)$ itself.
Therefore, given the degree $d_0$ of $v_0$, the messages $(\mu_{v\toboss}(t|\T(d)))_{v\in\partial v_0}$ are 
mutually independent Bernoulli variables with mean $p^{(t)}$.
Since the degree of $v_0$ is a Poisson variable with mean $d$, we conclude that
        $$p^{(t+1)}=\pr\br{\Po\br{dp^{(t)}}\geq k-1}=\phi_{d,k}\br{p^{(t)}}\qquad\mbox{for any }t\geq0.$$
Hence, \Lem~\ref{Prop_fix} implies that $\lim_{t\to\infty}p^{(t)}=p^*$.
Since 
the sequence $(\mu_{v_0\toboss}(t|\T(d)))_{t\geq1}$ is monotonically decreasing 
by \Lem~\ref{Lemma_WP}, the assertion follows from the monotone convergence theorem.
\end{proof}

\subsection{From random graphs to trees}

Given $t\geq0$ let $\T_{t}(d,k)$ be the random $\cbc{0,1}$-marked tree obtained from $\T(d)$ by marking each vertex $v$ with  $\mu_{v}(t|\T(d))\in\cbc{0,1}$ as defined in~\eqref{initialisation}-\eqref{marks}.
We recall the definition of $\Lambda_{d,k,n}$ from~(\ref{eqLambdadkn}).
The aim in this section is to prove

\begin{proposition}\label{Prop_main}
The limits
	\begin{align*}
	\theta_{d,k}&:=\lim_{t\to\infty}\cL([\T_t(d,k)]),&\hspace{-2cm}\Lambda_{d,k}&:=\lim_{n\to\infty}\Lambda_{d,k,n}
	\end{align*}
exist and $\Lambda_{d,k}=\atom_{\theta_{d,k}}$.
\end{proposition}

For a finite graph $G$ and $v\in V(G)$ let $\mu_{\nix}(t|G_v)$ denote the map $w\in V(G_v)\mapsto\mu_w(t|G_v)$.
Define
	\begin{align}\label{eqLambdadknt}
	\lambda_{k,G,t}&:=\frac1{|V(G)|}\sum_{v\in V(G)}\delta_{[G_v,v,\mu_{\nix}(t|G_v)]}\in\cP(\cG_{\cbc{0,1}}),&
	\Lambda_{d,k,n,t}&:=\Erw_{\G}[\atom_{\lambda_{k,\G,t}}]\in\cP^2(\cG_{\cbc{0,1}})
	\end{align}
(c.f.~\eqref{eqLambdadkn} and~\eqref{eqLambdadn}).
Thus, $\Lambda_{d,k,n,t}$ is the distribution of the neighbourhoods of the random graph $\G$ marked according to $\mu_{\nix}(t|\G)$.
The following lemma shows that the distribution of these marked neighbourhoods in $\G$ is described by
the distribution resulting from running Warning Propagation for $t$ rounds on the random tree $\T(d)$.

\begin{lemma}\label{Prop_WP}
We have \ $\lim_{n\to\infty}\Lambda_{d,k,n,t}=\atom_{\cL([\T_{t}(d,k)])}$ for any $t\geq0$.
\end{lemma}
\begin{proof}
Let $\tau$ be a rooted locally finite tree rooted at $v_0$.
Let $G$ be a graph and let $v\in V(G)$.
The construction of the Warning Propagation messages ensures that if $\partial^{s+t}[\tau,v_0]=\partial^{s+t}[G,v]$, then
	$\partial^{s}[\tau,v_0,\mu_{\nix}(t|\tau)]=\partial^{s}[G,v,\mu_{\nix}(t|G)]$.
Therefore, the assertion is immediate from \Thm~\ref{Thm_GW}.
\end{proof}

As a next step,  we show that running Warning Propagation for a bounded number $t$ of  rounds on $\G$ yields a very good approximation to the $k$-core.
Indeed, \Lem~\ref{Lemma_WP} shows that the $k$-core of $\G$ is contained in the set of all vertices $v$ such that $\mu_v(t|\G)=1$.
The following lemma, which is implicit in~\cite{MolloyCores,Riordan}, complements this statement.

\begin{lemma}\label{Prop_bottomUp}
For any $\eps>0$ there is $t>0$ such that
	$\abs{\cbc{v\in [n]:\mu_v(t|\G)=1}\setminus\core_k(\G)}\leq\eps n$ \whp\
\end{lemma}
\begin{proof}
Let 
	$Y_n=\frac{1}{n}|\core_k(\G)|$
be the fraction of vertices in the core.
We need to compare $Y_n$ with
	$X_n^{(t)}=\frac{1}{n}|\{v\in [n]:\mu_v(t|\G)=1\}|$, i.e., 
the fraction of vertices marked with $1$ after $t$ iterations of Warning Propagation.
Let
	$$x^{(t)}=\Erw_{\T(d)}\brk{\mu_{v_{0}}(t|\T(d))}$$
denote the probability that the root $v_0$ is marked with $1$ after $t$ iterations of Warning Propagation on $\T(d).$ Finally, set $p^{(0)}=1$ and let 
$p^{(t+1)}=\phi_{d,k}\br{p^{(t)}}$ (with $\phi_{d,k}$ from (\ref{eqmain})). 

We will establish the following relations:
\begin{equation}\label{eqOliver}
X_n(t) \;\; \sim_p 
	\;\; x^{(t)} \quad 
	\stacksign{$t\to\infty$}\longrightarrow \quad \phi_{d,k+1}(p^*)  \;\;
	\sim_p \;\; Y_n(t).
\end{equation}
The leftmost approximation (that $X_n(t)$ converges to $x^{(t)}$ in probability) is immediate from \Lem~\ref{Prop_WP}.
With respect to the second relation, we obtain from \Lem~\ref{Lemma_as} that
	\begin{equation}\label{eqMyInd}
	\Erw_{\T(d)}\brk{\mu_{v_0\toboss}(t|\T(d))}\stacksign{$t\to\infty$}\longrightarrow
			p^{*}.
	\end{equation}
Since 
each child of $v_0$ can be considered a root of an  independent instance of $\T(d)$ to which we can apply (\ref{eqMyInd}), 
we obtain
	\begin{equation}\label{eq:markprob}
	x^{(t+1)} \;\; = \;\; \pr\br{\sum_{u\in\partial v}\mu_{u\toboss}(t|\T(d))\geq k} \quad \stacksign{$t\to\infty$}\longrightarrow \quad \pr\br{\Po\br{dp^{*}}\geq k} \;\; = \;\; \phi_{d,k+1}(p^{*}).
	\end{equation}
Finally, Theorem \ref{Thm_PWS} implies 
the rightmost convergence in~(\ref{eqOliver}), whence the assertion is immediate.
\end{proof}

\begin{proof}[Proof of \Prop~\ref{Prop_main}]
We begin by proving that the sequence $(\cL([\T_t(d,k)]))_{t}$ of probability measures on $\cG_{\{0,1\}}$ converges.
Since $\cP(\cG_{\{0,1\}})$ is a Polish space, it suffices to show that for any bounded continuous function $f:\cG_{\cbc{0,1}}\to\RR$
the sequence 
	$(\Erw_{\T(d)}[f([\T_t(d,k)])])_t$ converges.
In fact, because the toplogy on $\cG_{\cbc{0,1}}$ is generated by the functions from~(\ref{eqtop}),
we may assume that $f=\chi_{\Gamma,s}$ for some $\Gamma\in\cG_{\cbc{0,1}}$, $s\geq0$.
Hence,
	\begin{align*}
	\Erw_{\T(d)}[f([\T_t(d,k)])]&=\pr\brk{\partial^s[\T_t(d,k)]=\partial^s\Gamma}.
	\end{align*}
To show that $(\pr\brk{\partial^s[\T_t(d,k)]=\partial^s\Gamma})_t$ converges, let $\eps>0$ and let $\cT_{s+1}$ be the $\sigma$-algebra generated
by the unmarked tree structure $\partial^{s+1}\T(d)$ up to distance $s+1$ from the root.
Let $\cB(s+1)$ be the set of vertices at distance precisely $s+1$ from the root.
The structure of the random tree $\T(d)$ is recursive, i.e., given $\cT_{s+1}$ the tree pending on each vertex $v\in\cB(s+1)$ is
just an independent copy of $\T(d)$ itself.
Therefore, \Lem~\ref{Lemma_as} and the union bound imply that conditioned on $\cT_{s+1}$ 
the limits
	$$\mu_{v\toboss}(\T(d))=\lim_{t\to\infty}\mu_{v\toboss}(t|\T(d))\qquad(v\in\cB(s+1))$$
exist almost surely.
In addition, conditioned on the vector $(\mu_{v\toboss}(\T(d)))_{v\in\cB(s+1)}$, all limits $\lim_{t\to\infty}\mu_{x}(t|\T(d))$ for
vertices $x$ at distance at most $s$ from $v_0$ are determined. 
Consequently, the sequence $(\pr\brk{\partial^s[\T_t(d,k)]=\partial^s\Gamma|\cT_{s+1}})_t$ of random variables converges almost surely.
Hence, so does the sequence
$(\pr\brk{\partial^s[\T_t(d,k)]=\partial^s\Gamma})_t$.
As this holds for any $\Gamma,s$, the limit $\theta_{d,k}=\lim_{t\to\infty}\cL([\T_t(d,k)])$ exists.

As a next step we show that $(\Lambda_{d,k,n})_{n}$ converges.
Because $\Lambda_{d,k,n}\in\cP^2(\cG_{\cbc{0,1}})$, which is a Polish space as well, it suffices to prove that 
$(\int f\dd\Lambda_{d,k,n})_{n}$ converges for any continuous function $f:\cP(\cG_{\cbc{0,1}})\to\RR$ with compact support. 
\Lem~\ref{Prop_WP} already shows that $(\Lambda_{d,k,n,t})_{n}$ converges for any $t$.
Hence, so does $(\int f\dd\Lambda_{d,k,n,t})_{n}$.
Therefore we will compare $\int f\dd\Lambda_{d,k,n}$ and $\int f\dd\Lambda_{d,k,n,t}$.
Plugging in the definitions of $\Lambda_{d,k,n}$ and $\Lambda_{d,k,n,t}$ (\eqref{eqLambdadkn} and~\eqref{eqLambdadknt}), we find that

	\begin{align*}
	\int f\dd\Lambda_{d,k,n}&=\Erw_{\G}\brk{f(\lambda_{k,\G})},&
	\int f\dd\Lambda_{d,k,n,t}&=\Erw_{\G}\brk{f(\lambda_{k,\G,t})}.
	\end{align*}
Hence,
	\begin{align}\label{eqconv666}
	\abs{\int f\dd\Lambda_{d,k,n}-\int f\dd\Lambda_{d,k,n,t}}\leq\Erw_{\G}\abs{f(\lambda_{k,\G})-f(\lambda_{k,\G,t})}.
	\end{align}
To bound the last term we will show that
\begin{align}\label{convp}
 \abs{\int\chi_{\Gamma,s}\dd\lambda_{k,\G}-\int\chi_{\Gamma,s}\dd\lambda_{k,\G,t}}\overset{t,n\to\infty}{\longrightarrow}0
\end{align}
in probability for all $\Gamma\in\cP(\cG_{\cbc{0,1}})$ and $s\geq 0$.
Plugging in the definitions of $\lambda_{k,\G}$ and $\lambda_{k,\G,t}$ we obtain
\begin{align*}
 \Erw_{\G}\abs{\int\chi_{\Gamma,s}\dd\lambda_{k,\G}-\int\chi_{\Gamma,s}\dd\lambda_{k,\G,t}}&
 \leq\Erw_{\G}\brk{\frac{1}{n}\sum_{v\in [n]}\abs{\int\chi_{\Gamma,s}\dd\atom_{\brk{\G_v,v,\sigma_{k,\G_v}}} - 
 \int\chi_{\Gamma,s}\dd\atom_{\brk{\G_v,v,\mu_{\nix}(t|G_v)}} }}\\
  &=\Erw_{\G,\vv}\left[\abs{\chi_{\Gamma,s}([\G_{\vv},\vv,\sigma_{k,\G_{\vv}}])-\chi_{\Gamma,s}([\G_{\vv},\vv,\mu_{\nix}(t,\G_{\vv})])}\right].
\end{align*}

By the definition of $\chi_{\Gamma,s}$ we have
	\begin{align}\label{eqconv667}
	\Erw_{\G,\vv}\left[\abs{\chi_{\Gamma,s}([\G_{\vv},\vv,\sigma_{k,\G_{\vv}}])-\chi_{\Gamma,s}([\G_{\vv},\vv,\mu_{\nix}(t,\G_{\vv})])}\right]&
		\leq\pr\brk{\partial^s[\G_{\vv},\vv,\sigma_{k,\G_{\vv}}]\neq\partial^s[\G_{\vv},\vv,\mu_{\nix}(t,\G_{\vv})]}.
	\end{align}
To bound the last term, let $\delta>0$ and assume that $n>n_0(\delta)$, $t>t_0(\delta)$ are sufficiently large.
Let $\cI(\ell,\G)$ be the set of all vertices $v\in[n]$ such that the number of vertices $u$ at distance at most $s$ exceeds $\ell$.
\Thm~\ref{Thm_GW} implies that there exists $\ell_0=\ell_0(\delta,d)$ such that $\pr\brk{|\cI(\ell_0,\G)|>\delta n}\leq\delta$.
Further, let $\cJ(\G)$ be the set of all $v\in[n]\setminus \cI(\ell_0,\G)$ such that $\mu_u(t|\G)=1$ but $u\not\in\cC_k(\G)$ for some vertex $u$
at distance at most $s$ from $v$.
Then \Lem~\ref{Prop_bottomUp} implies that $\pr\brk{|\cJ(\G)|>\delta n}\leq\delta$, provided that $n_0,t_0$ are sufficiently large.
Hence,
	\begin{align}\nonumber
	\pr\brk{\partial^s[\G_{\vv},\vv,\sigma_{k,\G_{\vv}}]\neq\partial^s[\G_{\vv},\vv,\mu_{\nix}(t,\G_{\vv})]}&\leq
		\pr\brk{|\cI(\ell_0,\G)|>\delta n}+\pr\brk{|\cJ(\G)|>\delta n}+\\
			&\qquad\pr\brk{\vv\in\cI(\ell_0,\G)\cup\cJ(\G)|
			|\cI(\ell_0,\G)|\leq\delta n,|\cJ(\G)|\leq\delta n}\leq4\delta\label{eqconv668}
	\end{align}
and we obtain (\ref{convp}). Now, let $\eps>0$. Since $\cP(\cG_{\cbc{0,1}})$ is a Polish space, it holds that $\cP(\cG_{\cbc{0,1}})$ is metrizable. 
Using this and the fact that the functions $\chi_{\Gamma,s}$ generate the topology on $\cG_{\cbc{0,1}}$, (\ref{convp}) implies that for given $\delta=\delta(\eps)>0$ 
there exist $n>n_0(\delta)$, $t>t_0(\delta)$ such that the distance of $\lambda_{\G,k}$ and 
$\lambda_{\G,k,t}$ is less than $\delta$ with probability larger than $1-\eps$. Since $f$ is uniformly continuous this implies
	\begin{align}\label{eqconv669}
	\Erw_{\G}\abs{f(\lambda_{k,\G})-f(\lambda_{k,\G,t})}<\eps.
	\end{align}
for suitable $\delta>0.$

Combining \eqref{eqconv666} and \eqref{eqconv669} and the first part of the proof, i.e.\ that $\theta_{d,k}=\lim_{t\to\infty}\cL([\T_t(d,k)])$ exists, 
and invoking \Lem~\ref{Prop_WP}, we conclude that 
	$$\lim_{n\to\infty}\Lambda_{d,k,n}=\lim_{t\to\infty}\lim_{n\to\infty}\Lambda_{d,k,n,t}=\atom_{\theta_{d,k}},$$
as desired.
\end{proof}

\section{The Branching Process}
\label{Sec_Kathrin}

\noindent
In this section we prove that in the limit $t\to\infty$ the random $\{0,1\}$-marked tree $\T_t(d,k)$ converges to the
$\cbc{0,1}$-marked tree $\T(d,k,p^*)$ produced by the $5$-type branching process $\hat \T(d,k,p^*)$ from \Sec~\ref{Sec_results}.
Together with \Lem s~\ref{Prop_WP} and~\ref{Prop_bottomUp} this will imply \Thm s~\ref{Thm_main} and~\ref{Thm_lwc}.

\subsection{Truncating the tree}
We begin by characterising the limiting distribution of the first few generations of $\T_t(d,k)$ as $t\to\infty$.
More precisely,  in the tree $\T_t(d,k)$ the vertices are marked by $\mu_v(t|\T(d))$.
By construction, these marks can be deduced from the messages $\mu_{w\toboss}(t|\T(d))$, $w\in V(\T(d))$ (cf.~(\ref{marks})).
The key feature of the messages $\mu_{w\toboss}(t|\T(d))$ is that they are solely determined by the tree pending on $w$.
That is, in contrast to the marks $\mu_v(t|\T(d))$, the messages $\mu_{w\toboss}(t|\T(d))$ are independent of the part of $\T(d)$ ``above'' $w$
and are
therefore much more convenient to work with.
On the other hand they also contain all the necessary information to compute the Warning Propagation marks $\mu_v(t|\T(d))$ on $\T(d)$.
We shall therefore begin by determining the limit as $t\to\infty$ of the distributions
	$$\theta_{d,k,t}^{s}:=\cL(\partial^s[\T(d),v_0,\mu_{\nix\toboss}(t|\T(d))]).$$
In words, this is the distribution resulting from the following experiment:
	create a random tree $\T(d)$ and mark each vertex with the message $\mu_{v\toboss}(t|\T(d))$.
Then, truncate the tree by deleting all vertices at distance greater than $s$ from the root.

What might $\theta_{d,k,t}^{s}$ converge to as $t\to\infty$?
If we assume that the point-wise limit of the messages $\mu_{v\toboss}(t|\T(d))$ as $t\to\infty$ exists, then the limit of $\theta_{d,k,t}^{s}$ should admit the following
simple description:
Once  we condition on the isomorphism class $[\partial^s\T(d)]$ of the tree up to level $s$, the messages $\lim_{t\to\infty}\mu_{u\toboss}(t|\T(d))$ for vertices
$u$ at distance less than $s$ from the root $v_0$ are determined by the \emph{boundary messages} 
$\lim_{t\to\infty}\mu_{v\toboss}(t|\T(d))$ sent out by the vertices at distance precisely $s$ from $v_0$.
Furthermore, each of these is governed by the tree pending on $v$ only.
These trees are mutually independent copies of $\T(d)$.
Thus, \Lem~\ref{Lemma_as} suggests that the ``boundary messages'' converge to a sequence of mutually independent $\Be(p^*)$ variables.
Consequently, the heuristically conjectured limiting distribution is the one obtained by creating the first $s$ levels of a random tree $\T(d)$, marking
each vertex at distance precisely $s$ by an independent $\Be(p^*)$ ``message'', and passing the messages up to the root.

To define this distribution formally, let 
$T$ be a locally finite tree rooted at $v_0$ and let $s>0$ be an integer.
Moreover, let $\beta=(\beta_w)_{w\in V(T)}$ be a family of independent $\Be(p^*)$ random variables.
If either $t=0$ or $v$ has distance greater than $s$ from $v_0$, we define
	\begin{equation}\label{eqIni}
	\mu_{v\toboss}^*(t|T,s)=\beta_v.
	\end{equation}
Moreover, if $t\geq0$ and if $v$ has distance less than or equal to $s$ from $v_0$, let
	\begin{equation}\label{eqIt1}
	\mu_{v\toboss}^*(t+1|T,s)=\vecone\cbc{\sum_{w\in\partial_+ v}\mu_{w\toboss}^*(t|T,s)\geq k-1}.
	\end{equation}
Let us denote the map $v\mapsto\mu_{v\toboss}^*(t|T,s)$ by $\mu_{\nix\toboss}^*(t|T,s)$.
Finally, define
	$$\theta_{d,k}^{s,*}:=\cL\bc{\partial^s[\T(d),v_0,\mu_{\nix\toboss}^*(s|\T(d),s)]}.$$
We begin with the following simple observation.

\begin{lemma}\label{claim_5}
Assume that $0<s<r$. We have $\cL(\partial^s[\T(d),v_0,\mu_{\nix\toboss}^*(s|\T(d),s)])=\cL(\partial^s[\T(d),v_0,\mu_{\nix\toboss}^*(r|\T(d),s)])$.
\end{lemma}
\begin{proof}
The assertion is immediate from the construction of the messages.
\end{proof}

\noindent
The main result of this subsection is

\begin{lemma}\label{Lemma_sizeless}
We have $\lim_{t\to\infty}\theta_{d,k,t}^{s}=\theta_{d,k}^{s,*}$ for all $s\geq 0$.
\end{lemma}
\begin{proof}
Let $0<\eps<1/10$.
We couple $\partial^s[\T(d),v_0,\mu_{\nix\toboss}(t|\T(d))]$ and $\partial^s[\T(d),v_0,\mu_{\nix\toboss}^*(s|\T(d),s)]$ such that both
operate on the same tree $\T(d)$.
Let $\cB$ be the set of all vertices of the random tree $\T(d)$ that have distance precisely $s$ from $v_0$.
Because $\Erw_{\T(d)}[|\cB|]$ is bounded, 
there exists $C=C(d,k,\eps)>0$ such that 
	\begin{equation}\label{eqsizeless1}
	\pr\brk{|\cB|\leq C}>1-\eps/2.
	\end{equation}
To prove the assertion, we are going to show that conditioned on $|\cB|\leq C$ there exists $t_0=t_0(\eps)$ such that for all $t>t_0$
there is a coupling of $\partial^s[\T(d),v_0,\mu_{\nix\toboss}(t|\T(d))]$ and $\partial^s[\T(d),v_0,\mu_{\nix\toboss}^*(s|\T(d),s)]$
such that both coincide with probability at least $1-\eps/2$.

Let $\beta=(\beta_v)_{v\in\cB}$ be a family of mutually independent $\Be(p^*)$ random variables.
Given the sub-tree $\partial^s[\T(d),v_0]$, the trees $\T_v$ pending on the vertices $v\in\cB$
are mutually independent and have the same distribution as the tree $\T(d)$ itself.
Therefore, \Lem~\ref{Lemma_as} implies that given $|\cB|\leq C$ there exist $t_1=t_1(d,k,\eps)$ 
and a coupling of $\beta$ with the trees $(\T_v)_{v\in\cB}$ such that
	\begin{equation}\label{eqsizeless2}
	\pr\brk{\mu_{v\toboss}(t|\T(d))=\beta_v	\forall t>t_1,v\in\cB	\; {\big |} \;|\cB|\leq C}>1-\eps/2.
	\end{equation}
Consider the event $\cE=\cbc{\mu_{v\toboss}(t|\T(d))=\beta_v \forall t>t_1,v\in\cB}$.

If the event $\cE$ occurs, then the initialisation $\mu_{v\toboss}^*(t|\T(d),s)=\beta_v$ for $v\in\cB$ (cf.\ (\ref{eqIni})),
and (\ref{eqIt1}) ensure that
	$$\mu_{u\toboss}^*(s|\T(d),s)=\mu_{u\toboss}^*(t|\T(d),s)=\mu_{u\toboss}(t|\T(d))\qquad\mbox{for all $t>t_1+s$}.$$
Hence, (\ref{eqsizeless1}) and (\ref{eqsizeless2}) yield
	$$\pr\brk{\partial^s[\T(d),v_0,\mu_{\nix\toboss}(t|\T(d))]=\partial^s[\T(d),v_0,\mu_{\nix\toboss}^*(s|\T(d),s)]}>1-\eps,$$
as desired.
\end{proof}

\subsection{Turning the tables}
The distribution $\theta_{d,k}^{s,*}$ describes the \emph{bottom-up} process of creating a random tree $\partial^s\T(d)$ to generation $s$,
generating a random boundary condition, and passing the messages up from the boundary to the root.
By contrast, the branching process from \Sec~\ref{Sec_results} proceeds in a \emph{top-down} fashion:
	the marks are created simultaneously with the tree.
We now construct a top-down process that produces the distribution $\theta_{d,k}^{s,*}$.

More precisely, define a random $\cbc{0,1}$-marked tree $\T^*(d,k)$ by means of the following two-type branching process (the type of a vertex $v$ will correspond to the message that $v$ passes to its parent).
Initially, there is a root vertex $v_0$ that has type $1$ with probability $p^*$ and type $0$ with probability $1-p^*$.
The offspring of a type $0$ vertex consist of $\Po(d(1-p^*))$ type $0$ vertices and independently $\Po_{<k-1}(dp^*)$ type $1$ vertices.
Further, a type $1$ vertex spawns $\Po(d(1-p^*))$ type $0$ offspring and independently $\Po_{\geq k-1}(d p^*)$ type $1$ offspring.
The mark of each vertex $v$, denoted by $\mu_{v\toboss}^*$, is identical to its type.
\begin{lemma}\label{Lemma_topdown}
For any $s\geq0$ we have $\cL\bc{\partial^s[\T^*(d,k)]}=\theta_{d,k}^{s,*}$.
\end{lemma}
\begin{proof}
Let us introduce the shorthands
	$$\cT(s,r)=\partial^{s}[\T(d),v_0,\mu_{\nix\toboss}^*(r|\T(d),s)],\qquad\cT(s)=\partial^s[\T^*(d,k)],$$
so our aim is to prove that $\cL(\cT(s,s))=\cL(\cT(s))$ for all $s$.
The proof is by induction on $s$.
In the case $s=0$ both $\cT(s)$ and $\cT(s,s)$ consist of the root $v_0$ only,
which is marked $1$ with probability $p^*$ and $0$ otherwise.

Now, assume that $\cL(\cT(s))=\cL(\cT(s,s))$.
To proceed to $s+1$, 
recall that 
the distribution $\cL(\cT(s+1,s+1))$ can be described as follows.
Create the random tree $\T(d)$ and let $\cB(r)$ be the set of vertices at distance precisely $r$ from the root for $r\geq0$.
Further, let $\beta_v=\mu_{v\toboss}^*(0|\T(d),s+1)$ for $v\in\cB(s+1)$.
Then $(\beta_v)_{v\in\cB(s+1)}$ is a family of independent $\Be(p^*)$ variables.
In addition, let $X_u(z)$ be the number of children $v$ of $u\in\cB(s)$ such that $\beta_v=z$.
Clearly, in the random tree $\T(d)$ the total number of children of $u\in\cB(s)$ has distribution $\Po(d)$, and these numbers are mutually independent
	conditioned on $\partial^s\T(d)$.
Since for each child $v$ we have $\beta_v=1$ with probability $p^*$ independently, we see that conditioned on $\partial^s\T(d)$ the random variables
	$(X_u(z))_{u\in\cB(s),z\in\cbc{0,1}}$ are mutually independent.
Moreover, $X_u(z)$ has distribution $\Po(dp^*)$ if $z=1$ and distribution $\Po(d(1-p^*))$ if $z=0$.
Further, $\mu_{u\toboss}^*(s+1|\T(d),s+1)=1$ iff $X_u(1)\geq k-1$.

Hence, the distribution of $\cT(s+1,s+1)$ conditioned on $\cT(s,s+1)$ can be described as follows.
Conditioned on $\cT(s,s+1)$, the random variables $(X_u(z))_{u\in\cB(s),z\in\cbc{0,1}}$ are mutually independent.
Furthermore, conditioned on $\mu_{u\toboss}^*(s+1|\T(d),s+1)=1$, $X_u(1)$ has distribution $\Po_{\geq k-1}(dp^*)$.
By contrast, given $\mu_{u\toboss}^*(s+1|\T(d),s+1)=0$, $X_u(1)$ has distribution $\Po_{<k-1}(dp^*)$.
In addition, $X_u(0)$ has distribution $\Po(d(1-p^*))$ for any $u$.
Therefore, the distribution of the random variables $(X_u(z))_{u\in\cB(s),z\in\cbc{0,1}}$ conditioned on $\cT(s,s+1)$ coincides
with the offspring distribution of the tree $\T^*(d,k)$.
Since $\cL(\cT(s,s+1))=\cL(\cT(s,s))=\cL(\cT(s))$ by \Lem~\ref{claim_5} and induction, the assertion follows.
\end{proof}

\subsection{Exchanging messages both ways}
\Lem s~\ref{Lemma_sizeless} and~\ref{Lemma_topdown} show that the labels $\mu_{v\toboss}^*$ of $\T^*(d,k)$ correspond to the
 ``upward messages'' that are sent toward the root in the tree $\T(d)$.
Of course, in the tree $\T(d)$ the marks $\mu_v(t|\T(d))$ 
can be computed from the messages $\mu_{v\toboss}(t|\T(d))$.
Indeed, for the root $v_0$ we simply have
	$$\mu_{v_0}(t|\T(d))=\vecone\cbc{\sum_{w\in\partial v_0}\mu_{w\to v_0}(t|\T(d))}.$$
However, for vertices $v\neq v_0$ there is a twist.
Namely, $\mu_v(t|\T(d))$ depends not only on the messages that $v$ receives from its children, but also on the message that its parent $u$ sends to $v$.
This message, in turn, depends on the message that $u$ receives from its parent, etc.\ up to the root.
Thus, we need to get a handle on the ``top-down'' message 
that $v$ receives from its parent.
These can be described recursively by letting
	\begin{equation}\label{eqTopDown1}
	\mu_{\fromboss v_0}(t|\T(d))=0   \qquad \mbox{for all }t\geq0,  
	\end{equation}
and for a vertex $v\neq v_0$ with parent $u$ we define
	\begin{align}\label{eqTopDown2}
	\mu_{\fromboss v}(0|\T(d))&=\mu_{u\to v}(0|\T(d)) =1 \nonumber\\
	\mu_{\fromboss v}(t+1|\T(d))&=\mu_{u\to v}(t+1|\T(d))
		=	\vecone\cbc{\mu_{\fromboss u }(t|\T(d))+
					\sum_{w\in\partial_+ u\setminus v}\mu_{w\toboss}(t|\T(d))\geq k-1}, %&\mbox{ if $v\neq u$ and $u$ is the parent of $v$}.
	\end{align}
where, as we recall, $\partial_+u$ is the set of children of $u$.
 Then
	\begin{align*}
	\mu_{v}(t|\T(d))&=\vecone\cbc{\mu_{\fromboss v}(t|\T(d))+\sum_{w\in\partial_+v}\mu_{w\toboss}(t|\T(d))\geq k}.
	\end{align*}
Let $\hat\T_t(d,k)$ signify the random $\labelset$-marked rooted tree obtained by marking each vertex of  $\T(d)$ with the triple
$(\mu_v(t|\T(d)),\mu_{v\toboss}(t|\T(d)),\mu_{\fromboss v}(t|\T(d)))$.

Our ultimate interest is in the marks $\mu_{v}(t|\T(d))$.
To get a handle on these, we are going to mimic the construction of the ``top-down'' messages
on the random tree $\T^*(d,k)$.
Of course, we set $\mu_{\fromboss v_0}^*=0$ 
and
$$\mu_{v_0}^*=\vecone\cbc{\sum_{w\in\partial v_0}\mu_{w\toboss}^*\geq k}.$$
Further, assume that $\mu_{\fromboss u}^*$ has been defined already and that $u$ is the parent of some vertex $v\neq v_0$.
Then we let 
	\begin{align}\label{eqKathrin1}
	\mu_{\fromboss v}^*&=\vecone\cbc{\mu_{\fromboss u}^*+\sum_{w\in\partial^+ u\setminus v}\mu_{w\toboss}^*\geq k-1},\\
	\mu_v^*&=\vecone\cbc{\mu_{\fromboss v}^*+\sum_{w\in\partial_+ v}\mu_{w\toboss}^*\geq k}.
                         \label{eqKathrin2}
	\end{align}
Let $\hat\T^*(d,k)$ signify the resulting tree in which each vertex is marked by the triple
	$(\mu_v^*,\mu_{v\toboss}^*,\mu_{\fromboss v}^*)$.
It is immediate from the construction that
	$$(\mu_v^*,\mu_{v\toboss}^*,\mu_{\fromboss v}^*)\in\labelset$$
for all $v$.

\begin{lemma}\label{Lemma_decorated}
For any $s>0$ we have \ $\lim_{t\to\infty}\cL(\partial^s[\hat\T_t(d,k)])=\cL(\partial^s[\hat\T^*(d,k)])$.
\end{lemma}
\begin{proof}
This is immediate from \Lem~\ref{Lemma_sizeless}, \Lem~\ref{Lemma_topdown} and the fact that
the definitions (\ref{eqTopDown1})--(\ref{eqTopDown2}) and (\ref{eqKathrin1}) 
of the ``top-down'' messages for $\hat\T_t(d,k)$ and $\hat\T^*(d,k)$ match.
\end{proof}

\subsection{Assembling the pieces}
Finally, we make the connection to the branching process from \Sec~\ref{Sec_results}.
Recall that $\hat\T(d,k,p^*)$ is the random $\cbc{0,1}^3$-marked tree produced by the $5$-type branching process with offspring distributions as in Figure~\ref{Fig_g}.

\begin{lemma}\label{Lem_id}
We have $\cL([\hat\T^*(d,k)])=\cL([\hat\T(d,k,p^*)])$.
\end{lemma}
\begin{proof}
It suffices to show that $\cL(\partial^s[\hat\T^*(d,k)])=\cL(\partial^s[\hat\T(d,k,p^*)])$ for any $s\geq0$.
The proof of this is by induction on $s$. 
Let $\cF_s$ be the $\sigma$-algebra generated by the $\cbc{0,1}$-marked tree $\partial^s \T^*(d,k)$.
That is, $\cF_s$ mirrors the information contained in the first $s$ generations of $\T^*(d,k)$, including the marks.
In addition, let $\hat\cF_s$ be the $\sigma$-algebra generated by $\partial^s\hat\T^*(d,k)$.
Then the construction of $\hat\T^*(d,k)$ ensures that
	\begin{equation}\label{Fact_measurable}
	\cF_s\subset\hat\cF_s\subset\cF_{s+1}\qquad\mbox{for any $s\geq0$.}
	\end{equation}

With respect to $s=0$, we see that $\mu_{\fromboss v_0}^*=0$ with certainty.
Moreover, $\mu_{v_0\toboss}^*$ has distribution $\Be(p^*)$ and $\mu_{v_0}^*=0$ if $\mu_{v_0\toboss}^*=0$.
On the other hand, conditioned on that $\mu_{v_0\toboss}^*=1$, 
the number $\sum_{w\in\partial v_0}\mu_{w\toboss}^*$ of children $w$ of $v_0$ in $\T^*(d,k)$  with $\mu_{w\toboss}^*=1$
has distribution $\Po_{\geq k-1}(dp^*)$, and $\mu_{v_0}^*=1$ iff $\sum_{w\in\partial v_0}\mu_{w\toboss}^*\geq k$.
Hence, using the fixed point property $p^*=\pr\br{\Po(dp^*)\geq k-1}$ and~\eqref{type_prob} with $q=q(d,k,p^*)$, we obtain
\begin{align*}
 \pr\br{(\mu_{v_0}^*,\mu_{v_0\toboss}^*,\mu_{\fromboss v_0}^*)=000}&=\pr\br{\Po(dp^*)<k-1}=1-p^*=p_{000},\\
 \pr\br{(\mu_{v_0}^*,\mu_{v_0\toboss}^*,\mu_{\fromboss v_0}^*)=010}&=\pr\br{\Po(dp^*)=k-1}=p^*q=p_{010},\\
\pr\br{(\mu_{v_0}^*,\mu_{v_0\toboss}^*,\mu_{\fromboss v_0}^*)=110}&=\pr\br{\Po(dp^*)\geq k}=p^*(1-q)=p_{110}.
\end{align*}

To proceed from $s$ to $s+1$, we condition on $\hat\cF_{s}$ and the aim is to derive the distribution of $\hat\T^*(d,k)$ given $\hat\cF_{s+1}$.
By~(\ref{Fact_measurable}) it is sufficient to study the random tree $\T^*(d,k)$ up to level $s+2$ conditioned on $\hat\cF_{s}$.
Thus, let $\cB_s$ be the set of all vertices $v$ of $\T^*(d,k)$ at distance precisely $s$ from the root.
Moreover, for each $v\in\cB_s$ let $\tau_v=(\tau_v(z_1,z_2,z_3))_{z_1,z_2,z_3\in\cbc{0,1}}$ 
be the number of children of $v$ marked $z_1z_2z_3$.
In addition, set
	$$\tau_v(z_2)=\sum_{z_1,z_3\in\cbc{0,1}}\tau_v(z_1,z_2,z_3).$$
Thus $\tau_v(z_2)$ is the number of messages of type $z_2$ that $v$ receives from its children.

By (\ref{Fact_measurable}) conditioned on $\hat\cF_s$ the random variables $(\tau_v)_{v\in\cB_s}$ are mutually independent and
the distribution of each individual $\tau_v$ is governed by the mark $(\mu_v^*,\mu_{v\toboss}^*,\mu_{\fromboss v}^*)$ only.
More precisely, 
we are going to verify that the distribution of $\tau_v$ is given by the generating function
	$g_{\mu_v^*,\mu_{v\toboss}^*,\mu_{\fromboss v}^*}$ from Figure~\ref{Fig_g} by  investigating the possible cases one by one.

We first observe that in all cases, the $\tau_v(0)$ has distribution $\Po(d(1-p^*))$ independently of the number of children of $v$ of all other types.
\begin{description}
\item[Case 1: $(\mu_v^*,\mu_{ v\toboss}^*,\mu_{\fromboss v}^*)=000$] 
	By (\ref{eqKathrin1}) we have $\mu_{\fromboss w}^*=0$ for all children $w$ of $v$.
	Further, since $\mu_{ v\toboss}^*=0$, we know that $\tau_v(1)<k-1$.
	Thus, $\tau_v(1)$ has distribution $\Po_{<k-1}(dp^*)$.
	Further, for a child $w$ of $v$, conditioned on $\mu_{w\toboss}^*=1$, we have $\mu_{w}^*=1$ iff
		$w$ has at least $k$ children $y$ such that $\mu_{y\toboss}^*=1$.
	This event occurs with probability $\pr\brk{\Po(dp^*)\geq k}$ independently for each $w$.
	Hence, conditioned on $\tau_v(1)$ we have $\tau(0,1,0)=\Bin(\tau_v(1),q)$ and $\tau_v(1,1,0)=\tau_v(1)-\tau(0,1,0)$.
	In summary, we obtain the generating function $g_{000}$.
\item[Case 2: $(\mu_v^*,\mu_{ v\toboss}^*,\mu_{\fromboss v}^*)=001$] 
	There are two sub-cases.
	\begin{description}
	\item[Case 2a: $\tau_v(1)=k-2$]
		then for any child $w$ of $v$ we have $\mu_{\fromboss w}^*=1-\mu_{w\toboss}^*$.
		Hence, for each of the $k-2$ children $w$ such that $\mu_{w\toboss}^*=1$ we have $\mu_w^*=1$ 
		iff $w$ has at least $k$ children $y$ such that $\mu_{y\toboss}^*=1$.
		Thus, $\mu_w^*=\Be(1-q)$ independently for each such $w$.
		Moreover, for each child $w$ of $v$ with $\mu_{w\toboss}^*=0$ we have $\mu_w^*=0$.
	\item[Case 2b: $\tau_v(1)<k-2$]
		We have $\mu_{\fromboss w}^*=0$ for all children $w$ of $v$.
		Hence, $\tau_v(0)$ has distribution $\Po(d(1-p^*))$ and for every child $w$ with $\mu_{w\toboss}^*=0$ we have $\mu_w^*=0$.
		Thus, $\tau_v(0)=\tau_v(0,0,0)$.
		Further, $\tau_v(1)$ has distribution $\Po_{<k-2}(dp^*)$.
		Finally, since $\mu_{\fromboss w}^*=0$ for all $w$, any child $w$ such that $\mu_{w\toboss}^*=1$ satisfies $\mu_{w}^*=1$ iff
		$w$ has at least $k$ children $y$ such that $\mu_{y\toboss}^*=1$.
		This event occurs with probability $1-q$ independently for all such $w$.
	\end{description}
	Since the first sub-case occurs with probability $\bar q$ and the second one accordingly with probability $1-\bar q$, we obtain
	the generating function $g_{001}$.
\item[Case 3: $(\mu_v^*,\mu_{v\toboss}^*,\mu_{\fromboss v}^*)=010$] 
	Because $\mu_v^*=\mu_{\toboss v}^*=0$, we have $\tau_v(1)=k-1$ with certainty.
	Further, because $\mu_{\fromboss v}^*=0$ and $\tau_v(1)=k-1$, (\ref{eqKathrin1}) entails that $\mu_{\fromboss w}^*=1-\mu_{w\toboss}$
		for all children $w$ of $v$.
	Hence, if $w$ is a child such that $\mu_{w\toboss}^*=1$, then $\mu_w^*=1$ iff $w$ has at least $k$ children $y$ such that $\mu_{y\toboss}^*=1$.
	This event occurs with probability $1-q$ for each $w$ independently.
	Consequently, $\tau(0)=\tau(0,0,1)$ and $\tau_v(1)=\tau_v(1,1,0)+\tau_v(0,1,0)$ and $\tau_v(1,1,0)=\Bin(\tau_v(1),1-q)$.
	Thus, the offspring distribution of $v$ is given by $g_{010}$.
\item[Case 4: $(\mu_v^*,\mu_{ v\toboss}^*,\mu_{\fromboss v}^*)=110$] 
	Since $\mu_v^*=1$, (\ref{eqKathrin1}) entails that $\mu_{\fromboss w}^*=1$ for all children $w$ of $v$.
	Hence $\tau_v(0)=\tau_v(0,0,1)$.
	Moreover, 	since $\mu_{\fromboss v}^*=0$ and $\mu_v^*=1$, (\ref{eqKathrin2}) implies that $\tau_v(1)=\tau_v(1,1,1)\geq k$.
	Consequently, $\tau_v(1)=\Po_{\geq k}(dp^*)$ independently of $\tau_0(v)$.
	Thus, we obtain $g_{110}$.
\item[Case 5: $(\mu_v^*,\mu_{ v\toboss}^*,\mu_{\fromboss v}^*)=111$] 
	As in the previous case, $\mu_v^*=1$, (\ref{eqKathrin1}) ensures that $\mu_{\fromboss w}^*=1$ for all children $w$ of $v$.
	Thus, $\tau_v(0)=\tau_v(0,0,1)$.
	Furthermore, as $\mu_v^*=\mu_{\fromboss v}^*=1$, $\tau_v(1)=\tau_v(1,1,1)$ has distribution $\Po_{\geq k-1}(dp^*)$.
	In summary, the distribution of the offspring of $v$ is given by $g_{111}$.
\end{description}
Thus, in each case we obtain the desired offspring distribution.
\end{proof}

\begin{proof}[Proof of \Thm~\ref{Thm_lwc}]
By \Prop~\ref{Prop_main} we have
	$$\lim_{n\to\infty}\Lambda_{d,k,n}=\atom_{\theta_{d,k}}\quad \text{and}\quad \theta_{d,k}=\lim_{t\to\infty}\cL([\T_t(d,k)]).$$

Moreover, combining \Lem s~\ref{Lemma_sizeless}, \ref{Lemma_topdown}, \ref{Lemma_decorated} and~\ref{Lem_id}, we see that
$\theta_{d,k}=\thet_{d,k,p^*}$.
\end{proof}

\begin{proof}[Proof of \Thm~\ref{Thm_main}]
We deduce \Thm~\ref{Thm_main} from \Thm~\ref{Thm_lwc}.
Let  $s\geq0$ and let $\tau$ be a $\cbc{0,1}$-marked rooted tree.
The function
	$f:\cG_{\cbc{0,1}}\to\RR,\ \gamma\mapsto\chi_{\tau,s}(\gamma)$
is continuous, where $\chi_{\Gamma,s}$ is as defined in~\eqref{eqtop}, and we let
	$$z=\Erw_{\T(d,k,p^*)}[f([\T(d,k,p^*)])]=\pr\brk{\partial^s[\T(d,k,p^*)]=\partial^s[\tau]}.$$
The function
	$$F:\cP^2(\cG_{\cbc{0,1}})\to\RR,\qquad\xi\mapsto\int\abs{\int f\dd\nu-z}\dd\xi(\nu),$$
where, of course, $\nu$ ranges over $\cP(\cG_{\cbc{0,1}})$, is continuous as well.
Consequently, \Thm~\ref{Thm_lwc} implies that
	\begin{align}\label{eqConvProb1}
	\lim_{n\to\infty}\int F\dd\Lambda_{d,k,n}&=
		\abs{\int f\dd \thet_{d,k,p^*}-z}=|\Erw_{\T(d,k,p^*)}[f([\T(d,k,p^*)])]-z|=0.
	\end{align}
Let
	$X_\tau(\G)=n^{-1}\abs{\cbc{v\in[n]:\partial^s[\G_v,v,\sigma_{k,\G_v}]=\partial^s[\tau]}}.$
Plugging in the definition of $\Lambda_{d,k,n}$, we obtain
	\begin{align}\label{eqConvProb2}
	\int F\dd\Lambda_{d,k,n}&=
		\Erw_{\G}\abs{X_\tau(\G)-z}.
	\end{align}
Finally, combining (\ref{eqConvProb1}) and~(\ref{eqConvProb2}) completes the proof.
\end{proof}

\paragraph{\bf Acknowledgment}
The first author thanks Victor Bapst and Guilhem Semerjian for helpful discussions.

\section{Appendix: Tables of Definitions}\label{appendix}

We provide reference tables of various definitions which we have made throughout the paper. We sometimes give only informal descriptions here -- the precise definitions appear in the main body of the paper.

\subsection{Random Trees}
$\T(d)$ is the standard (unmarked single-type) Galton-Watson tree in which each vertex has $\Po(d)$ children independently. From $\T(d)$ we may construct labels \emph{bottom-up} using a variant of Warning Propagation.

\vspace{0.5cm}

\noindent \textbf{Bottom-up Trees:}
\\
\\
\begin{tabular}{lll}
Tree name & Types & Further description\\
$\T_t(d,k)$ & $\{0,1\}$ & Obtained from $\T(d)$ after $t$ rounds of Warning Propagation.\\
$\hat{\T}_t(d,k)$ & \labelset & Obtained from $\T(d)$ after $t$ rounds of 5-type Warning Propagation.\\
\end{tabular}

\vspace{0.5cm}

Alternatively, we may construct labels \emph{top-down}, so the labels are constructed simultaneously with the tree.

\vspace{0.5cm}

\noindent \textbf{Top-down Trees}\\
\\
\begin{tabular}{lll}
Tree name & Types & Further description\\
$\hat{\T}(d,k,p)$ & $\labelset$ & Constructed according to the generating functions of Figure~\ref{Fig_g}.\\
$\T(d,k,p)$ & $\{0,1\}$ & $2$-type projection of $\hat{\T}(d,k,p)$.\\
\end{tabular}

\vspace{0.5cm}

Note that there are certain consistent notational conventions: $\hat{\T}$ indicates a $5$-type tree, while $\T_t$ indicates a tree whose labels were created bottom-up using some variant of Warning Propagation. Finally, we have two more trees which allow us in a sense to transition between the top-down and the bottom-up trees.

\vspace{0.5cm}

\noindent \textbf{Transition Trees}\\
\\
\begin{tabular}{lll}
Tree name & Types & Further description\\
$\T^*(d,k)$ & $\{0,1\}$ & Labels created top-down, mimic the upwards messages $\mu_{v\toboss}$.\\
$\hat{\T}^*(d,k)$ & \labelset & Obtained from $\T^*(d,k)$ according to the rules~\eqref{eqKathrin1} and~\eqref{eqKathrin2}.
\end{tabular}

\vspace{0.5cm}

Lemma~\ref{Lem_id} says that $\hat{\T}^*(d,k)$ has the same distribution as $\hat{\T}(d,k,p^*)$.

\subsection{Distributions}
We define various probability distributions and their corresponding laws in the paper which we list here, including some equivalences which are not part of the definitions, but which we prove during the course of the paper.

\noindent \textbf{Distributions in $\cP(\cG)$, resp. $\cP(\cG_{\{0,1\}})$:}\\
\\
\begin{tabular}{lll}
Distribution & Definition & Description\\
$\lambda_{G}$ & $\frac1{|V(G)|}\sum_{v\in V(G)}\delta_{[G_v,v]}$ & distribution of neighbourhoods of vertices in $G$.\\
$\lambda_{k,G}$ & $\frac1{|V(G)|}\sum_{v\in V(G)}\delta_{[G_v,v,\sigma_{k,G_v}]}$ & vertices labelled according to membership of the core.\\
$\lambda_{k,G,t}$ & $\frac1{|V(G)|}\sum_{v\in V(G)}\delta_{[G_v,v,\mu_{\nix}(t|G_v)]}$ & vertices labelled after $t$ rounds of Warning Propagation.\\
\end{tabular}

\vspace{0.5cm}

\noindent \textbf{Distributions in $\cP^2(\cG)$, resp. $\cP^2(\cG_{\{0,1\}})$:} \\
\\
\begin{tabular}{lll}
Distribution & Definition & Description\\
$\Ldn$ & $\Erw_{\G}[\delta_{\lambda_{\G}}]$ & distribution of neighbourhoods of the random graph \eqref{eqLambdadn}.\\
$\Lambda_{d,k,n}$ & $\Erw_{\G}[\delta_{\lambda_{k,\G}}]$ & vertices labelled according to membership of the core \eqref{eqLambdadkn}.\\
$\Lambda_{d,k,n,t}$ & $\Erw_{\G}[\lambda_{k,\G,t}]$ & vertices labelled after $t$ rounds of Warning Propagation \eqref{eqLambdadknt}.\\
$\Lambda_{d,k}$ & $\lim_{n\to \infty} \Lambda_{d,k,n} $ & limiting labelled neighbourhood distribution of the random graph (Prop.~\ref{Prop_main}).\\
\end{tabular}

\vspace{0.5cm}

\noindent \textbf{Distribution laws:}\\
\\
\begin{tabular}{lll}
Distribution Law & Definition & Remarks\\
$\thet_{d,k,p}$ & $\cL[\T(d,k,p)]$ & \\
$\theta_{d,k}$ & $\lim_{t\to \infty}\cL[\T_t(d,k)]$ & = $\thet_{d,k,p^*}$ \; (Prop.~\ref{Prop_main}, Thm.~\ref{Thm_lwc}).\\
$\theta_{d,k,t}^{s}$ & $\cL(\partial^s[\T(d),v_0,\mu_{\nix\toboss}(t|\T(d))])$ & \\
$\theta_{d,k}^{s,*}$ & $\cL(\partial^s[\T(d),v_0,\mu_{\nix\toboss}^*(s|\T(d),s)])$ & $ = \lim_{t\to \infty}\theta_{d,k,t}^{s} = \cL(\partial^s [\T^*(d,k)])$ \; (Lemmas~\ref{Lemma_sizeless} and~\ref{Lemma_topdown}).\\
\end{tabular}

\vfill

\end{document}